\documentclass[amstex,12pt, amssymb]{article}

\usepackage{mathtext}
\usepackage[cp1251]{inputenc}
\usepackage[T2A]{fontenc}
\usepackage[dvips]{graphicx}
\usepackage{amsmath}
\usepackage{amssymb}
\usepackage{amsxtra}
\usepackage{latexsym}
\usepackage{ifthen}

\newcounter{lemma}[section]

\newcounter{corollary}[section]

\newcounter{remark}[section]

\newcounter{theorem}[section]

\newcounter{proposition}[section]

\numberwithin{equation}{section}

\begin{document}

\markboth{\centerline{E. SEVOST'YANOV}} {\centerline{ON THE
OPENNESS...}}

\def\cc{\setcounter{equation}{0}
\setcounter{figure}{0}\setcounter{table}{0}}

\overfullrule=0pt


\author{{E. SEVOST'YANOV}\\}

\title{
{\bf ON OPENNESS AND DISCRETENESS OF MAPPINGS SATISFYING ONE
INEQUALITY WITH RESPECT TO $p$-MODULUS}}

\date{\today}
\maketitle

\begin{abstract}
A paper is devoted to study of topological properties of some class
of space mappings. It is showed that, sense preserving mappings
$f:D\rightarrow \overline{{\Bbb R}^n}$ of a domain $D\subset{\Bbb
R}^n,$ $n\geqslant 2,$ satisfying some modulus inequality with
respect to $p$-modulus of families of curves, are open and discrete
at some restrictions on a function $Q,$ which determinate inequality
mentioned above.
\end{abstract}


\section{Preliminaries}

The present paper is devoted to the study of quasiconformal mappings
and their generalizations, such as mappings with finite distortion
intensively investigated last time, see 
\cite{BGMV}, \cite{BGR}, \cite{Cr$_1$}--\cite{Cr$_2$},
\cite{Gol$_1$}--\cite{Gol$_2$}, \cite{IM}, \cite{IR},
 \cite{KO}, 
\cite{MRSY$_1$}--\cite{MRSY$_2$}, \cite{Mikl}, \cite{Pol},
\cite{Re$_1$}--\cite{Re$_2$}, \cite{Ri}, 
\cite{RSY$_1$}--\cite{RSY$_2$}, \cite{UV},
\cite{Va$_1$}--\cite{Va$_2$}.

Let us give some definitions. Everywhere below, $D$ is a domain in
${\Bbb R}^n,$ $n\ge 2,$ $m$ is the Lebesgue measure in ${\Bbb R}^n,$
$m(A)$ the Lebesgue measure of a measurable set $A\subset {\Bbb
R}^n.$  A mapping $f:D\rightarrow {\Bbb R}^n$ is {\it discrete} if
$f^{-1}(y)$ consists of isolated points for each $y\in{\Bbb R}^n,$
and $f$ is {\it open} if it maps open sets onto open sets. The
notation $f:D\rightarrow {\Bbb R}^n$ assumes that $f$ is continuous.
A mapping $f$ is said to be {\it orientation preserving,} if the
topological index $µ(y, f,G)>0$ for an arbitrary domain $G\subset D$
such that $\overline{G}\subset D$ and $y\in f(G)\setminus f(\partial
G),$ see e.g. \cite[II.2]{Re$_2$}. Given a mapping $f:D\rightarrow
{\Bbb R}^n,$ a set $E\subset D$ and a point $y\in {\Bbb R}^n,$ we
define the multiplicity function $N(y, f, E)$ as the number of
pre-images of $y$ in $E,$ i.e., $$N(y, f, E) = {\rm card}\,\left\{x
\in E: f(x) = y\right\}\,$$ and
$$N(f,E)=\sup\limits_{y\in{\Bbb
R}^n}\,N(y,f,E)\,.$$
A set $H\subset \overline{{\Bbb R}^n}$ is called {\it totally
disconnected}, if every it's component degenerate to a point; in
this case we write ${\rm dim\,}H=0,$ where ${\rm dim}$ denotes {\it
a topological dimension} of $H$ (see \cite{HW}). A mapping
$f:D\rightarrow \overline{{\Bbb R}^n}$ is said to be {\it light}, if
${\rm dim\,}\{f^{\,-1}(y)\}=0$ for every $y\in \overline{{\Bbb
R}^n}.$ Set
$$B(x_0, r)=\left\{x\in{\Bbb R}^n: |x-x_0|< r\right\},\quad {\Bbb
B}^n := B(0, 1)\,, \quad {\Bbb S}^{n-1}:=S(0, 1)\,,$$
$\Omega_n$ is a volume of the unit ball ${\Bbb B}^n$ in ${\Bbb
R}^n,$ and $\omega_{n-1}$ is an area of the unit sphere ${\Bbb
S}^{n-1}$ in ${\Bbb R}^n.$

\medskip
A curve $\gamma$ in ${\Bbb R}^n$ is a continuous mapping $\gamma
:\Delta\rightarrow{\Bbb R}^n$ where $\Delta$ is an interval in
${\Bbb R} .$ Its locus $\gamma(\Delta)$ is denoted by $|\gamma|.$
Given a family $\Gamma$ of curves $\gamma$ in ${\Bbb R}^n ,$ a Borel
function $\rho:{\Bbb R}^n \rightarrow [0,\infty]$ is called {\it
admissible} for $\Gamma ,$ abbr. $\rho \in {\rm adm}\, \Gamma ,$ if
$$\int\limits_{\gamma} \rho(x)|dx| \ge 1$$ for each (locally
rectifiable) $\gamma\in\Gamma.$ Given $p\ge 1,$ the {\it
$p$--mo\-du\-lus} of $\Gamma$ is defined as the quantity
$$M_p(\Gamma):=\inf\limits_{ \rho \in {\rm adm}\, \Gamma}
\int\limits_{{\Bbb R}^n} \rho^p(x) dm(x)$$ interpreted as $+\infty$
if ${\rm adm}\, \Gamma = \varnothing .$ Note that
$M_p(\varnothing)=0;$ $M_p(\Gamma_1)\le M_p(\Gamma_2)$ whenever
$\Gamma_1\subset\Gamma_2,$ and
$M_p\left(\bigcup\limits_{i=1}^{\infty}\Gamma_i\right)\le
\sum\limits_{i=1}^{\infty}M_p(\Gamma_i),$ see
\cite[Theorem~6.2]{Va$_1$}.

\medskip
Denote $\Gamma(E,F,D)$ a family of all paths
$\gamma:[a,b]\rightarrow \overline{{\Bbb R}^n},$ which join $E$ and
$F$ in $D,$ i.e., $\gamma(a)\in E,\gamma(b)\in\,F$ and $\gamma(t)\in
D$ при $t \in (a, b).$

\medskip
A following fact was established in our recent paper \cite{Sev$_1$}.
Let $f$ be a mapping of a domain $D\subset{\Bbb R}^n,$ $n\geqslant
2,$ into ${\Bbb R}^n$ obeying a condition
\begin{equation} \label{eq2*B}
M(\Gamma )\leqslant \int\limits_{f(D)} Q(y)\cdot \rho_*^n (y) dm(y)
\end{equation}
for every $\rho_*\in {\rm adm}\,f(\Gamma)$ with respect to a
conformal modulus $M(\Gamma):=M_n(\Gamma)$ and a given function
$Q:{\Bbb R}^n\rightarrow [0, \infty],$ $Q(x)\equiv 0$ for all $x\in
{\Bbb R}^n\setminus f(D).$ Then $f$ is open and discrete whenever
$Q$ satisfies some conditions. Given $y_0\in f(D)$ and numbers
$0<r_1<r_2<\infty,$ we denote
\begin{equation}\label{eq1**}
A(r_1,r_2,y_0)=\left\{ y\,\in\,{\Bbb R}^n:
r_1<|y-y_0|<r_2\right\}\,.\end{equation}
A goal of the present paper is to prove an analogous result in the
situation, when $n-1<p\leqslant n.$ Namely, given $y_0\in f(D)$ and
$0<r_1<r_2<\infty,$ denote $\Gamma(y_0, r_1, r_2)$ a family of paths
in $D$ such that $f(\Gamma)\in \Gamma(S(y_0, r_1), S(y_0, r_2),
A(r_1,r_2,y_0)).$ Instead of (\ref{eq2*B}), assume that $f$
satisfies the inequality
\begin{equation} \label{eq2*A}
M_p(\Gamma(y_0, r_1, r_2))\leqslant \int\limits_{f(D)} Q(y)\cdot
\eta^p (y) dm(y)
\end{equation}
for some $p\in (n-1, n],$ every $y_0\in f(D),$ every
$0<r_1<r_2<\infty,$ and every nonnegative Lebesgue measurable
function $\eta: (r_1,r_2)\rightarrow [0,\infty ]$ with
\begin{equation}\label{eqA2}
\int\limits_{r_1}^{r_2}\eta(r) dr\geqslant 1\,.
\end{equation}
Observe that the inequality (\ref{eq2*A}) is more weaker than
(\ref{eq2*B}) even at $p=n.$ In fact, let $\rho_*\in {\rm
adm}\,f(\Gamma),$ and assume that the relation (\ref{eq2*B}) holds.
Let $\eta: (r_1,r_2)\rightarrow [0,\infty ]$ be a Lebesgue
measurable function which satisfies (\ref{eqA2}). Set
$\rho_*(y):=\eta(|y-y_0|),$ $A=A(r_1,r_2,y_0),$ $S_1=S(y_0, r_1)$
and $S_2=S(y_0, r_2).$ Observe that $\rho_*\in\Gamma(S_1, S_2, A)$
since
$\int\limits_{\gamma}\,\rho_{*}(y)\,|dy|\,\geqslant\int\limits_{r_1}
^{r_2}\,\eta(t)dt\geqslant 1$ for every $\gamma\in \Gamma(S_1, S_2,
A)$ by \cite[теорема~5.7]{Va$_1$}.
Consequently, we can substitute $\rho_*$ in (\ref{eq2*B}), whence we
obtain (\ref{eq2*A}).

\medskip
The present paper is devoted to investigation of the following
question:

\medskip
{\it What kind of connection between the discreteness of $f$ and the
estimate (\ref{eq2*A}) at some $n-1<p\leqslant n$ ?}

\medskip
An answer is given bellow. As we noted above, a case $p=n$ is
studied in \cite{Sev$_1$} earlier .

Let $Q:D\rightarrow [0,\infty]$ be a Lebesgue measurable function,
then  $q_{x_0}(r)$ denotes an integral average of $Q(x)$ under the
sphere $S(x_0, r),$
\begin{equation}\label{eq32*}
q_{x_0}(r):=\frac{1}{\omega_{n-1}r^{n-1}}\int\limits_{|x-x_0|=r}Q(x)\,dS\,,
\end{equation}
where $dS$ -- denotes an element of area of $S.$ We say that a
function ${\varphi}:D\rightarrow{\Bbb R}$ has {\it a finite mean
oscillation } at a point $x_0\in D$, write $\varphi\in FMO(x_0),$ if
%
$$\overline{\lim\limits_{\varepsilon\rightarrow 0}}\ \
\frac{1}{\Omega_n\varepsilon^n} \ \ \int\limits_{B(
x_0,\,\varepsilon)}
|{\varphi}(x)-\overline{{\varphi}}_{\varepsilon}|\ dm(x)\, <\,
\infty,
$$
%
where
%
$\overline{{\varphi}}_{\varepsilon}\,=\,\frac{1}{\Omega_n\varepsilon^n}\int\limits_{B(
x_0,\,\varepsilon)} {\varphi}(x)\ dm(x)$ (see e.g.
\cite[section~6.1]{MRSY$_2$}). A main result of the present paper is
the following.

\medskip
\begin{theorem}\label{th3}{\sl\,Let $Q:{\Bbb R}^n\rightarrow (0,
\infty)$ be a Lebesgue measurable function equal to 0 outside of
$f(D),$ and let $f:D\,\rightarrow\,{\Bbb R}^n$ be a sense-preserving
mapping obeying (\ref{eq2*A}) for every $y_0\in f(D),$ every
$0<r_1<r_2<\infty,$ some $p\in (n-1, n]$ and every nonnegative
Lebesgue measurable function $\eta: (r_1,r_2)\rightarrow [0,\infty
]$ obeying (\ref{eqA2}). Then $f$ is discrete and open whenever the
function $Q$ satisfies at least one of the following conditions:

1) $Q\in FMO(y_0)$ for every $y_0\in f(D),$

2) $q_{y_0}(r)\,=\,O\left(\left[\log{\frac1r}\right]^{n-1}\right)$
as $r\rightarrow 0$ and every $y_0\in f(D),$ where a function
$q_{y_0}(r)$ is defined by (\ref{eq32*}),

3) for every $y_0\in f(D)$ there exists $\delta(y_0)>0$ such that
for sufficiently small $\varepsilon>0$
\begin{equation}\label{eq5**}
\int\limits_{\varepsilon}^{\delta(y_0)}\frac{dt}{t^{\frac{n-1}{p-1}}q_{y_0}^{\frac{1}{p-1}}(t)}<\infty,
\qquad
\int\limits_{0}^{\delta(y_0)}\frac{dt}{t^{\frac{n-1}{p-1}}q_{y_0}^{\frac{1}{p-1}}(t)}=\infty\,.
\end{equation}
}
\end{theorem}

\medskip
\begin{remark}\label{rem3}
Theorem \ref{th3} holds for mappings
$f:D\,\rightarrow\,\overline{{\Bbb R}^n},$ also. In this case, we
need require the conditions 1)--3) at $y_0=0$ for
$\widetilde{f}=f\circ\varphi,$ where $\varphi(x)=\frac{x}{|x|^2},$
$\varphi:\infty\mapsto 0.$
\end{remark}

\medskip
\section{Main lemma}
\setcounter{equation}{0}

{\it A continuum} is called a connected compactum $C\subset
\overline{{\Bbb R}^n}.$ We say that a family of paths $\Gamma_1$
{\it is minorized} by a family $\Gamma_2,$ write $\Gamma_1 >
\Gamma_2,$
if for every $\gamma \in \Gamma_1$ there exists a subpath which
belongs to $\Gamma_2.$
In this case, $M_p(\Gamma_1) \leqslant M_p(\Gamma_2)$ (see
\cite[Theorem~6.4]{Va$_1$}).

\medskip
Now we need some information from a theory of general metric spaces.
Let $(X, \mu)$ be a metric space with measure $\mu.$ For each real
number $n\ge 1,$ we define {\it the Loewner function} $\phi_n:(0,
\infty)\rightarrow [0, \infty)$ on $X$ as
$$\phi_n(t)=\inf\{M_n(\Gamma(E, F, X)): \Delta(E, F)\leqslant t\}\,,$$
where $E$ and $F$ are disjoint nondegenerate continua in $X$ with
$$\Delta(E, F):=\frac{{\rm dist}\,(E,
F)}{\min\{{\rm diam\,}E, {\rm diam\,}F\}}\,.$$
A pathwise connected metric measure space $(X, \mu)$ is said to be a
{\it Loewner space} of exponent $n,$ or an $n$-Loewner space, if the
Loewner function $\phi_n(t)$ is positive for all $t> 0$ (see
\cite[section~2.5]{MRSY$_2$} or \cite[Ch.~8]{He}). Observe that
${\Bbb R}^n$ and ${\Bbb B}^n\subset {\Bbb R}^n$ are Loewner spaces
(see \cite[Theorem~8.2 and Example~8.24(a)]{He}). As known, a
condition $\mu(B(x_0, r))\geqslant C\cdot r^n$ holds in Loewner
spaces $X$ for every point $x_0\in X,$ for some constant $C$ and all
$r<{\rm diam}\,X.$ A space $X$ is called {\it geodesic}, if every
pair of points in $X$ can be joined by a curve whose length is equal
to the distance between the points. In particular, ${\Bbb B}^n$ is a
geodesic space.  A following definition can be found in
\cite[section~1.4, ch.~I]{He} or \cite[section~1]{AS}. A measure
$\mu$ in a metric space is called doubling if balls have finite and
positive measure and there is a constant $C(\mu)\geqslant 1$ such
that $\mu(B(x_0, 2r))\le C\cdot \mu(B(x_0, r))$ for every $r>0$ and
every $x_0\in X.$ We also call a metric measure space $(X, \mu)$
{\it doubling} if $\mu$ is a doubling measure. Following
\cite[section~7.22]{He}, given a real-valued function $u$ in a
metric space $X,$ a Borel function $\rho\colon X\rightarrow [0,
\infty]$ is said to be an upper gradient of $u$ if
$|u(x)-u(y)|\leqslant \int\limits_{\gamma}\rho\,|dx|$ for each
rectifiable curve $\gamma$ joining $x$ and $y$ in $X.$ Let $(X,
\mu)$ be a metric measure space and let $1\leqslant p<\infty.$ We
say that $X$ admits {\it a $(1; p)$-Poincare inequality} if there is
constant $C\geqslant 1$ so that
$$\frac{1}{\mu(B)}\int\limits_{B}|u-u_B|d\mu(x)\leqslant C\cdot({\rm
diam\,}B)\left(\frac{1}{\mu(\tau B)} \int\limits_{\tau B}\rho^n
d\mu(x)\right)^{1/n}$$
for all balls $B$ in $X,$ for all bounded continuous functions $u$
on $B,$ and for all upper gradients $\rho$ of $u.$ Metric measure
spaces where formula
$$\frac{1}{C}R^{n}\leqslant \mu(B(x_0,
R))\leqslant CR^{n}$$
holds for every $x_0\in X,$ some constant $C\geqslant 1$ and all
$R<{\rm diam}\,X$ are called {\it Ahlfors $n$-regular.}

\medskip
A following statement holds.

\medskip
\begin{proposition}\label{pr1}
The unit ball ${\Bbb B}^n$ is Ahlfors $n$-regular metric space, in
which $(1; n)$-Poincare inequality holds. Moreover, the estimate
\begin{equation}\label{eq3}
M_p(\Gamma(E, F, {\Bbb B}^n))>0
\end{equation}
holds for any continua $E, F\subset {\Bbb B}^n$ and every $p\in
(n-1, n].$
\end{proposition}

\begin{proof}
By comments given above, the unit ball ${\Bbb B}^n$ is Ahlfors
$n$-regular, moreover, a space ${\Bbb B}^n$ is geodesic and is a
Loewner space. By \cite[Theorems~9.8 and 9.5]{He},  $(1;
n)$-Poincare inequality holds in  ${\Bbb B}^n.$ In this case,
(\ref{eq3}) holds by \cite[Corollary~4.8]{AS}.
\end{proof}$\Box$

\medskip Let $A(\varepsilon,
\varepsilon_0, y_0)$ be defined by (\ref{eq1**}) at
$r_1=\varepsilon$ and $r_2=\varepsilon_0.$ A following lemma
includes the main result of the present paper in the most general
situation.

\medskip
\begin{lemma}\label{lem1}
{\sl\,Let $Q:{\Bbb R}^n\rightarrow (0, \infty)$ be a Lebesgue
measurable function equal to 0 outside of $f(D),$ and let
$f:D\,\rightarrow\,{\Bbb R}^n$ be a sense-preserving mapping obeying
(\ref{eq2*A}) for every $y_0\in f(D),$ every $0<r_1<r_2<\infty,$
some $p\in (n-1, n]$ and every nonnegative Lebesgue measurable
function $\eta: (r_1,r_2)\rightarrow [0,\infty ]$ obeying
(\ref{eqA2}).

If, for every $y_0\in D$ and some $\varepsilon_0>0,$
\begin{equation} \label{eq4!}
\int\limits_{A(\varepsilon, \varepsilon_0,
y_0)}Q(y)\cdot\psi^p(|y-y_0|) \ dm(y)\,=\,o\left(I^p(\varepsilon,
\varepsilon_0)\right)
\end{equation}
where $\psi(t):(0,\infty)\rightarrow [0,\infty]$ is some nonnegative
Lebesgue measurable function such that
\begin{equation} \label{eq5}
0< I(\varepsilon,
\varepsilon_0):=\int\limits_{\varepsilon}^{\varepsilon_0}\psi(t)dt <
\infty\qquad \forall \varepsilon\in(0, \varepsilon_0)\,,
\end{equation}
then $f$ is open and discrete.}
\end{lemma}

\medskip
\begin{remark}\label{rem1}
In the conditions of Lemma \ref{lem1}, we can consider that
$\int\limits_{\varepsilon}^{A}\,\psi(t) dt>0$ for every
$\varepsilon\in (0, A)$ and some fixed $A$ with $0<A<\varepsilon_0.$
In fact, the integral (\ref{eq4!}) increase at decreasing of
$\varepsilon.$ Now, since $Q(x)>0$ a.e., we obtain from (\ref{eq4!})
and (\ref{eq5}) that $\int\limits_{\varepsilon}^{A}\,\psi(t)
dt\rightarrow \infty$ as $\varepsilon\rightarrow 0.$
\end{remark}

\medskip
{\it A proof of Lemma \ref{lem1}.} Without loss of generality, we
can consider that $D={\Bbb B}^n.$ Since every light sense-preserving
mapping $f:D\rightarrow {\Bbb R}^n$ is open and discrete in $D,$ see
e.g., \cite[Corollary, p.~333]{TY}, it is sufficiently to prove that
$f$ is light. Suppose a contrary.   Then there exists $y_0\in {\Bbb
R}^n$ such that a set $\{f^{\,-1}(y_0)\}$ is not totally
disconnected. Now, there exists a continuum $C\subset
\{f^{\,-1}(y_0)\}.$ Since $f$ is a sense-preserving, $f\not\equiv
y_0.$ By theorem on a preserving of a sign, there exists $x_0\in D$
and $\delta_0>0:$ $\overline{B(x_0, \delta_0)}\subset D$ and
\begin{equation}\label{eq7*}
f(x)\ne y_0\qquad 
\forall\quad x\in \overline{B(x_0, \delta_0)}\,.
\end{equation}
By \cite[Lemma~1.15]{Na} at $p=n,$ and Proposition \ref{pr1} at
$p\in (n-1, n),$
\begin{equation}\label{eq4*}
M_p\left(\Gamma\left(C, \overline{B(x_0, \delta_0)}, {\Bbb
B}^n\right)\right)>0\,.
\end{equation}
By (\ref{eq7*}), since $f(C)=\{y_0\},$ every path of
$\Delta=f\left(\Gamma\left(C, \overline{B(x_0, \delta_0)}, {\Bbb
B}^n\right)\right)$ does not degenerate to a point. From other hand,
an endpoint of every path of $\Delta$ is $y_0.$ Let $\Gamma_i$ be a
family of paths $\alpha_i(t):(0,1)\rightarrow {\Bbb R}^n$ such that
$\alpha_i(1)\in S(y_0, r_i),$ $r_i<\varepsilon_0,$ $r_i$ is some
strictly positive sequence with $r_i\rightarrow 0$ as $i\rightarrow
\infty,$ and $\alpha_i(t)\rightarrow y_0$ as $t\rightarrow 0.$ Now
\begin{equation}\label{eq12*}
\Gamma\left(C, \overline{B(x_0, \delta_0)}, {\Bbb B}^n\right) =
\bigcup\limits_{i=1}^\infty\,\, \Gamma_i^*\,,
\end{equation}
where $\Gamma_i^*$ is a subfamily of all paths $\gamma$ from
$\Gamma\left(C, \overline{B(x_0, \delta_0)}, {\Bbb B}^n\right),$
such that $f(\gamma)$ has a subpath in $\Gamma_i.$ Observe that
\begin{equation}\label{eq8*}
\Gamma_i^*>\Gamma(\varepsilon, r_i, y_0)
\end{equation}
for every $\varepsilon\in (0, r_i).$
Set
$$\eta_{i,\varepsilon}(t)=\left\{
\begin{array}{rr}
\psi(t)/I(\varepsilon, r_i), & t\in (\varepsilon, r_i)\,,\\
0,  &  t\not\in (\varepsilon, r_i)\,,
\end{array}
\right. $$
where $I(\varepsilon, r_i)=\int\limits_{\varepsilon}^{r_i}\,\psi (t)
dt.$ Observe that
$\int\limits_{\varepsilon}^{r_i}\eta_{i,\varepsilon}(t)dt=1.$ Now we
can apply (\ref{eq2*A}). By (\ref{eq2*A}) and (\ref{eq8*}),
\begin{equation}\label{eq11*}
M_p(\Gamma_i^*)\leqslant M_p(\Gamma(r_i, \varepsilon, y_0))\leqslant
\int\limits_{A(\varepsilon, \varepsilon_0, y_0)} Q(y)\cdot
\eta^p_{i,\varepsilon}(|y-y_0|)dm(y)\,\leqslant {\frak
F}_i(\varepsilon),
\end{equation}
where
${\frak F}_i(\varepsilon)=\,\frac{1}{{I(\varepsilon,
r_i)}^p}\int\limits_{A(\varepsilon, \varepsilon_0,
y_0)}\,Q(y)\,\psi^{p}(|y-y_0|)\,dm(y)$ и $I(\varepsilon,
r_i)=\int\limits_{\varepsilon}^{r_i}\,\psi (t) dt.$ By (\ref{eq4!}),
$$\int\limits_{A(\varepsilon, \varepsilon_0, y_0)}\,Q(y)\,\psi^{p}(|y-y_0|)\,dm(y)\,=\,
G(\varepsilon)\cdot\left(\int\limits_{\varepsilon}^{\varepsilon_0}\,\psi
(t) dt\right)^p\,,$$
where $G(\varepsilon)\rightarrow 0$ as $\varepsilon\rightarrow 0$ by
assumption of Lemma. Observe that
${\frak F}_i(\varepsilon)\,=\,G(\varepsilon)\cdot\left(1\,+\,\frac{
\int\limits_{r_i}^{\varepsilon_0}\,\psi(t)
dt}{\int\limits_{\varepsilon}^{r_i}\,\psi(t) dt}\right)^p,$
where $\int\limits_{r_i}^{\varepsilon_0}\,\psi(t) dt<\infty$ is a
fixed real number, and $\int\limits_{\varepsilon}^{r_i}\,\psi(t)
dt\rightarrow \infty$ as $\varepsilon\rightarrow 0,$ because left
hand-part of (\ref{eq4!}) increase by a decreasing of $\varepsilon.$
Thus, ${\frak F}_i(\varepsilon)\rightarrow 0.$ Taking a limit in
(\ref{eq11*}) as $\varepsilon\rightarrow 0,$ left hand-part of which
does not depend on $\varepsilon,$ we obtain that $M_p(\Gamma_i^*)=0$
for all positive integer $i.$ However, in this case, by
(\ref{eq12*}) and from that,
$M_p\left(\bigcup\limits_{i=1}^{\infty}\Gamma_i\right)\leqslant
\sum\limits_{i=1}^{\infty}M_p(\Gamma_i)$
(\cite[Theorem~6.2]{Va$_1$}), we obtain: $M_p\left(\Gamma\left(C,
\overline{B(x_0, \delta_0)}, {\Bbb B}^n\right)\right) =0.$ The last
contradicts to (\ref{eq4*}). The contradiction obtained above proves
that $f$ is light. Consequently, by \cite[Coroollary, p.~333]{TY},
$f$ is open and discrete. $\Box$

\section{Proofs of main results} Now we show that {\it a proof of Theorem \ref{th3}}
follows from Lemma \ref{lem1}. In the case 1), when $Q\in FMO(y_0),$
consider a function $\psi(t)=\left(t\,\log{\frac1t}\right)^{-n/p}>0$
for which we apply Lemma \ref{lem1}. By \cite[Corollary~6.3,
Ch.~6]{MRSY$_2$}, we obtain
\begin{equation}\label{eq31*}
\int\limits_{\varepsilon<|y-y_0|<\varepsilon_0}Q(y)\cdot\psi^p(|y-y_0|)
\ dm(y)\,=\,O \left(\log\log \frac{1}{\varepsilon}\right)\,,\quad
\varepsilon\rightarrow 0\,.
\end{equation}
for $\varepsilon<\varepsilon_0$ and some $\varepsilon_0>0.$ For
$I(\varepsilon, \varepsilon_0),$ defined in Lemma \ref{lem1}, we
have
\begin{equation}\label{eqlogest}
I(\varepsilon,
\varepsilon_0)=\int\limits_{\varepsilon}^{\varepsilon_0}\psi(t) dt
>\log{\frac{\log{\frac{1}
{\varepsilon}}}{\log{\frac{1}{\varepsilon_0}}}}.
\end{equation}
Now, by (\ref{eq31*}),
$$ \frac{1}{I^p(\varepsilon,
\varepsilon_0)}\int\limits_{\varepsilon<|y-y_0|<\varepsilon_0}
Q(x)\cdot\psi^p(|y-y_0|)dm(y)\leqslant
C\left(\log\log\frac{1}{\varepsilon}\right)^{1-p}\rightarrow 0,
\quad \varepsilon\rightarrow 0\,,
$$
that yields a desired conclusion in a case 1), because
(\ref{eq4!})--(\ref{eq5}) are satisfied. Observe that a case 2) is a
consequence of 3), thus, we can restrict us by consideration of a
case 3). In this case, set
\begin{equation}\label{eq9}
I=I(\varepsilon,\varepsilon_0)=\int\limits_{\varepsilon}^{\varepsilon_0}\
\frac{dr}{r^{\frac{n-1}{p-1}}q_{y_0}^{\frac{1}{p-1}}(r)}\,.
\end{equation}
Given $0<\varepsilon<\varepsilon_0<1,$ set
\begin{equation}\label{eq1*****}
\psi(t)= \left \{\begin{array}{rr}
1/[t^{\frac{n-1}{p-1}}q_{y_0}^{\frac{1}{p-1}}(t)]\ , & \ t\in
(\varepsilon,\varepsilon_0)\ ,
\\ 0\ ,  &  \ t\notin (\varepsilon,\varepsilon_0)\ .
\end{array} \right.
\end{equation}
Observe that $\psi$ satisfies all of the conditions of Lemma
\ref{lem1}.
By Fubini Theorem (\cite[Theorem~8.1, Ch.~III]{Sa}),
$\int\limits_{\varepsilon<|y-y_0|<\varepsilon_0}
Q(y)\cdot\psi^p(|y-y_0|)\, dm(y)=\omega_{n-1}\cdot I(\varepsilon,
\varepsilon_0)$ (where $\omega_{n-1}$ is an area of the unit sphere
of ${\Bbb S}^{n-1}$ in ${\Bbb R}^n$). We conclude that
(\ref{eq4!})--(\ref{eq5}) hold, that complete the proof. $\Box$

\section{Corollaries} The following statements can be obtained from Lemma \ref{lem1}.

\medskip
\begin{corollary}\label{cor1}
{\sl\,Let $Q:{\Bbb R}^n\rightarrow (0, \infty)$ be a Lebesgue
measurable function equal to 0 outside of $f(D),$ and let
$f:D\,\rightarrow\,{\Bbb R}^n$ be a sense-preserving mapping obeying
(\ref{eq2*A}) for every $y_0\in f(D),$ every $0<r_1<r_2<\infty,$
some $p\in (n-1, n]$ and every nonnegative Lebesgue measurable
function $\eta: (r_1,r_2)\rightarrow [0,\infty ]$ obeying
(\ref{eqA2}). Then $f$ is discrete and open whenever the function
$Q$ satisfies the following conditions:

1)
$\int\limits_{\varepsilon}^{\delta_0}\frac{dt}{tq_{y_0}^{\frac{1}{n-1}}(t)}<\infty$
for every $y_0\in f(D),$ some $\delta_0=\delta_0(y_0)$ and small
enough $\varepsilon>0,$ and

2)
$\int\limits_{0}^{\delta_0}\frac{dt}{tq_{y_0}^{\frac{1}{n-1}}(t)}=\infty$
for every $y_0\in f(D)$ and some $\delta_0=\delta_0(y_0).$ }
\end{corollary}

\medskip\begin{proof}
Set  $\psi(t)\quad=\quad \left \{\begin{array}{rr}
\left(1/[tq^{\frac{1}{n-1}}_{y_0}(t)]\right)^{n/p}\ , & \ t\in
(\varepsilon, \varepsilon_0)\ ,
\\ 0\ ,  &  \ t\notin (\varepsilon,
\varepsilon_0)\ ,
\end{array} \right.$ Arguing similarly with the proof of a case 3) of the Theorem
\ref{th3}, we obtain a desired conclusion.
\end{proof}$\Box$

\medskip
Let now $p\in (n-1, n).$

\medskip
\begin{corollary}\label{cor3}
{\sl\, Let $Q:{\Bbb R}^n\rightarrow (0, \infty)$ be a Lebesgue
measurable function equal to 0 outside of $f(D),$ and let
$f:D\,\rightarrow\,{\Bbb R}^n$ be a sense-preserving mapping obeying
(\ref{eq2*A}) for every $y_0\in f(D),$ every $0<r_1<r_2<\infty,$
some $p\in (n-1, n]$ and every nonnegative Lebesgue measurable
function $\eta: (r_1,r_2)\rightarrow [0,\infty ]$ obeying
(\ref{eqA2}). Then $f$ is discrete and open whenever the function
$Q$ satisfies the condition $Q\in L_{loc}^s({\Bbb R}^n)$ at some
$s\geqslant\frac{n}{n-p}.$ }
\end{corollary}

\medskip
\begin{proof}
Fix $0<\varepsilon_0<\infty.$ Given $y_0\in f(D),$ set $G:=B(y_0,
\varepsilon_0)$ and $\psi(t):=1/t.$ Observe that $\psi$ satisfies
(\ref{eq5}). It remains to verify (\ref{eq4!}) now. Applying
H\"{o}lder inequality, we obtain
\begin{equation}\label{eq13}
\int\limits_{\varepsilon<|x-b|<\varepsilon_0}\frac{Q(x)}{|x-b|^p} \
dm(x)\leqslant
\left(\int\limits_{\varepsilon<|x-b|<\varepsilon_0}\frac{1}{|x-b|^{pq}}
\ dm(x) \right)^{\frac{1}{q}}\,\left(\int\limits_{G}
Q^{q^{\prime}}(x)\ dm(x)\right)^{\frac{1}{q^{\prime}}}\,,
\end{equation}
where  $1/q+1/q^{\prime}=1$. Observe that the integral in the left
hand-part of (\ref{eq13}) can be directly calculated. In fact, let
$q^{\prime}=\frac{n}{n-p}$ (and, consequently, $q=\frac{n}{p}.$) By
Fubini theorem,
$$
\int\limits_{\varepsilon<|x-b|<\varepsilon_0}\frac{1}{|x-b|^{pq}} \
dm(x)=\omega_{n-1}\int\limits_{\varepsilon}^{\varepsilon_0}
\frac{dt}{t}=\omega_{n-1}\log\frac{\varepsilon_0}{\varepsilon}\,.
$$
Following to notions made in Lemma \ref{lem1}, we obtain that
$$
\frac{1}{I^p(\varepsilon,
\varepsilon_0)}\int\limits_{\varepsilon<|x-b|<\varepsilon_0}\frac{Q(x)}{|x-b|^p}
\ dm(x)\leqslant \omega^{\frac{p}{n}}_{n-1}\Vert
Q\Vert_{L^{\frac{n}{n-p}}(G)}\left(\log\frac{\varepsilon_0}{\varepsilon}\right)
^{-p+\frac{p}{n}}\,\rightarrow 0\,,
$$
as $\varepsilon\rightarrow 0,$ which implies (\ref{eq4!}).

\medskip
Now, let $q^{\prime}>\frac{n}{n-p}$ (
$q=\frac{q^{\prime}}{q^{\prime}-1}$). In this case,
$$
\int\limits_{\varepsilon<|x-b|<\varepsilon_0}\frac{1}{|x-b|^{pq}} \
dm(x) = \omega_{n-1}\int\limits_{\varepsilon}^{\varepsilon_0}
t^{n-\frac{pq^{\prime}}{q^{\prime}-1}-1}dt\leqslant
\omega_{n-1}\int\limits_{0}^{\varepsilon_0}
t^{n-\frac{pq^{\prime}}{q^{\prime}-1}-1}dt
=\frac{\omega_{n-1}}{n-\frac{pq^{\prime}}{q^{\prime}-1}}\varepsilon^{n-\frac{pq^{\prime}}{q^{\prime}-1}}_0,
$$
and, consequently,
$$
\frac{1}{I^p(\varepsilon, \varepsilon_0)}
\int\limits_{\varepsilon<|x-b|<\varepsilon_0}\frac{Q(x)}{|x-b|^p} \
dm(x)\leqslant \Vert
Q\Vert_{L^{q^{\prime}}(G)}\left(\frac{\omega_{n-1}}{n-\frac{pq^{\prime}}{q^{\prime}-1}}
\varepsilon^{n-\frac{pq^{\prime}}{q^{\prime}-1}}_0\right)^{\frac{1}{q}}\left(\log\frac{\varepsilon_0}{\varepsilon}\right)^{-p}\,,
$$
which implies (\ref{eq4!}). Now a desired conclusion follows from
Lemma \ref{lem1}.
\end{proof}$\Box$

\section{Examples}\setcounter{equation}{0}

First of all, let us give some examples of mappings obeying
(\ref{eq2*B}) and (\ref{eq2*A}). It is known that, for an arbitrary
quasiregular mapping $f:D \rightarrow {\Bbb R}^n,$ one has
$$M(\Gamma)\le N(f, A)K_O(f)M(f(\Gamma))$$
for some constant $K_O(f)\geqslant 1,$ for an arbitrary Borel set
$A$ in the domain $D$ such that $N(f, A)<\infty$ and an arbitrary
family $\Gamma$ of curves $\gamma$ in $A$ (see
\cite[Theorem~3.2]{MRV$_1$} or \cite[Theorem~6.7, Chap.~II]{Ri}).

\medskip
Observe that $N(y,f,A)$ is Lebesgue measurable for any Borel
measurable set $A$ (see \cite[Theorem of section~IV.1.2]{RR}).

Now, to give some another examples.  Set at points $x\in D$ of
differentiability of $f$
$$\Vert
f^{\,\prime}(x)\Vert\,=\,\max\limits_{h\in {\Bbb R}^n \backslash
\{0\}} \frac {|f^{\,\prime}(x)h|}{|h|}\,, J(x,f)=\det
f^{\,\prime}(x),$$
and define for any $x\in D$ and $p\geqslant 1$
$$K_{O, p}(x,f)\quad =\quad \left\{
\begin{array}{rr}
\frac{\Vert f^\prime(x)\Vert^p}{|J(x,f)|}, & J(x,f)\ne 0,\\
1,  &  f^{\,\prime}(x)=0, \\
\infty, & {\rm otherwise}
\end{array}
\right.\,.$$
We say that a property $P$ holds for {\it $p$-almost every
($p$-a.e.)} curves $\gamma$ in a family $\Gamma$ if the subfamily of
all curves in $\Gamma $, for which $P$ fails, has $p$-mo\-du\-lus
zero.
Recall that a mapping $f:D\rightarrow {\Bbb R}^n$ is said to have
{\it $N$-pro\-per\-ty (by Luzin)} if
$m\left(f\left(S\right)\right)=0$ whenever $m(S)=0$ for
$S\subset{\Bbb R}^n.$ Similarly, $f$ has the {\it
$N^{-1}$-pro\-per\-ty} if $m\left(f^{\,-1}(S)\right)=0$ whenever
$m(S)=0.$

If $\gamma :\Delta\rightarrow{\Bbb R}^n$ is a locally rectifiable
curve, then there is the unique nondecreasing length function
$l_{\gamma}$ of $\Delta$ onto a length interval $\Delta
_{\gamma}\subset{\Bbb R}$ with a prescribed normalization $l
_{\gamma}(t_0)=0\in\Delta _{\gamma},$ $t_0\in\Delta,$ such that $l
_{\gamma}(t)$ is equal to the length of the subcurve $\gamma
|_{[t_0,t]}$ of $\gamma$ if $t>t_0,$ $t\in\Delta ,$ and $l
_{\gamma}(t)$ is equal to minus length of $\gamma |_{[t,t_0]}$ if
$t<t_0,$ $t\in\Delta .$ Let $g: |\gamma |\rightarrow{\Bbb R}^n$ be a
continuous mapping, and suppose that the curve $\widetilde{\gamma}
=g\circ\gamma$ is also locally rectifiable. Then there is a unique
non--decreasing function $L_{\gamma ,g}: \Delta
_{\gamma}\rightarrow\Delta _{\widetilde{\gamma}}$ such that
$L_{\gamma ,g}\left(l_{\gamma}(t)\right) =
l_{\widetilde{\gamma}}(t)$ for all $t\in\Delta.$ We say that a
mapping $f:D\rightarrow{\Bbb R}^n$ is {\it absolutely continuous on
paths with respect to $p$-modulus,} write $f\in ACP_p,$ if for
$p$-a.e. curve $\gamma:\Delta\rightarrow D$ the function $L_{\gamma
,f}$ is locally absolutely continuous on all closed intervals of
$\Delta.$

\medskip
A following result is a insignificant amplification of one classical
result for quasiregular mappings (see \cite[Theoreme~2.4,
Ch.~II]{Ri}).

\medskip
\begin{theorem}\label{th1}
{\sl Let a mapping $f:D\rightarrow {\Bbb R}^n$ be differentiable
a.e. in $D,$ have $N$- and $N^{-1}$-pro\-per\-ties, and
$ACP_p$-pro\-per\-ty for some $p\geqslant 1,$ too. Let $A$ in $D$ be
a Borel set, and let $\Gamma$ be a family of paths in $A.$ Assume
that $q:{\Bbb R}^n\rightarrow [0, \infty]$ is a Borel function
obeying $K_{O, p}(x,f)\leqslant q(f(x))$ for a.e. $x\in D.$ Then
relation
$$M_{p}(\Gamma)\le \int\limits_{{\Bbb R}^n}\rho^{\,\prime p}(y)N(y, f, A)q(y)dm(y)$$
holds for every $\rho^{\,\prime}\in {\rm adm\,}f(\Gamma).$}
\end{theorem}

\begin{proof}
Let $\rho^{\,\prime}\in {\rm adm\,}f(\Gamma).$ Set
$\rho(x)=\rho^{\,\prime}(f(x))\Vert f^{\,\prime}(x)\Vert$ for $x\in
A$ and $\rho(x)=0$ otherwise. Let $\Gamma_0$ be a family of all
locally rectifiable curves of $\Gamma,$ where $f$ is absolutely
continuous. Since $f\in ACP_p,$ we obtain that
$M_p(\Gamma)=M_p(\Gamma_0).$ Now, by \cite[Lemma~2.2, Ch.~II]{Ri},
$\int\limits_{\gamma} \rho(x) |dx|=\int\limits_{\gamma}
\rho^{\,\prime}(f(x))\Vert f^{\,\prime}(x)\Vert |dx|\ge
\int\limits_{f\circ\gamma} \rho^{\,\prime}(y) |dy| \ge 1,$ and.
consequently, $\rho\in {\rm adm\,}\Gamma_0.$ By change of variables
formula for differentiable mappings, which have $N$ and
$N^{\,-1}$-properties (see \cite[Proposition~8.3]{MRSY$_2$}),
$$M_{p}(\Gamma)=M_{p}(\Gamma_0)\le \int\limits_{{\Bbb R}^n}
\rho^p(x)\,dm(x)=\int\limits_{A} \frac{\rho^{\,\prime\,p}(f(x))\Vert
f^{\,\prime}(x)\Vert^p J(x, f)}{J(x ,f)}\, dm(x)\le$$
$$
\le\int\limits_{A} \rho^{\,\prime\,p}(f(x))q(f(x))|J(x, f)|\,
dm(x)=\int\limits_{{\Bbb R}^n}\rho^{\,\prime p}(y)N(y, f,
A)q(y)dm(y)\,.$$ Here we take into account, that $J(x, f)\ne 0$
a.e., see \cite[Proposition~8.3]{MRSY$_2$}. Theorem is proved.
\end{proof}$\Box$

\medskip
From theorem \ref{th1}, we obtain the following.

\medskip
\begin{corollary}\label{cor2}
{\sl Let a mapping $f:D\rightarrow {\Bbb R}^n$ be differentiable
a.e. in $D,$ have $N$- and $N^{-1}$-pro\-per\-ties, and $f\in
W_{loc}^{1, p}$ for some $p\geqslant 1,$ too. Let $A$ in $D$ be a
Borel set, and let $\Gamma$ be a family of paths in $A.$ Assume that
$q:{\Bbb R}^n\rightarrow [0, \infty]$ is a Borel function obeying
$K_{O, p}(x,f)\leqslant q(f(x))$ for a.e. $x\in D.$ Then relation
$$M_{p}(\Gamma)\le \int\limits_{{\Bbb R}^n}\rho^{\,\prime p}(y)N(y, f, A)q(y)dm(y)$$
holds for every $\rho^{\,\prime}\in {\rm adm\,}f(\Gamma).$}
\end{corollary}

\begin{proof}
As known, $W_{loc}^{1, p}=ACL^p$ (see \cite[Theorems~1 and 2,
section~1.1.3]{Maz}). However, $ACL^p\subset ACP_p$ by Fuglede's
Lemma (see \cite[Theorem~28.2]{Va$_1$}). The rest follows from the
Theorem \ref{th1}.
\end{proof}$\Box$

\medskip
From theorems \ref{th3} and \ref{th1}, and from Corollaries
\ref{cor1} and \ref{cor3}, we obtain the following.

\medskip
\begin{corollary}\label{cor4}{\sl\,Let $p\in (n-1, n],$ and
let a mapping $f:D\rightarrow {\Bbb R}^n$ be differentiable a.e. in
$D,$ have $N$- and $N^{-1}$-pro\-per\-ties, and $ACP_p$-pro\-per\-ty
for some $p\geqslant 1,$ too. Let $A$ in $D$ be a Borel set, and let
$\Gamma$ be a family of paths in $A.$ Assume that $q:{\Bbb
R}^n\rightarrow [0, \infty]$ is a Borel function obeying $K_{O,
p}(x,f)\leqslant q(f(x))$ for a.e. $x\in D.$ Let $Q:{\Bbb
R}^n\rightarrow (0, \infty)$ be a Lebesgue measurable function
defined as follows: $Q(y)=N(y, f, D)\cdot\max\{q(y), 1\}$ at $y\in
f(D),$ and $Q(y)\equiv 1$ for $y\in {\Bbb R}^n\setminus f(D).$
Assume that the function $Q$ satisfies at least one of the following
conditions:

1) $Q\in FMO(y_0)$ for every $y_0\in f(D),$

2) $q_{y_0}(r)\,=\,O\left(\left[\log{\frac1r}\right]^{n-1}\right)$
as $r\rightarrow 0,$ for every $y_0\in f(D),$ where $q_{y_0}(r)$ is
defined by (\ref{eq32*}).

3) for every $y_0\in f(D)$ there exists $\delta(y_0)>0$ such that
for sufficiently small $\varepsilon>0$
$\int\limits_{\varepsilon}^{\delta(y_0)}\frac{dt}{t^{\frac{n-1}{p-1}}q_{y_0}^{\frac{1}{p-1}}(t)}<\infty$
and, besides that,
$\int\limits_{0}^{\delta(y_0)}\frac{dt}{t^{\frac{n-1}{p-1}}q_{y_0}^{\frac{1}{p-1}}(t)}=\infty.$

4)
$\int\limits_{\varepsilon}^{\delta_0}\frac{dt}{tq_{y_0}^{\frac{1}{n-1}}(t)}<\infty$
for every $y_0\in f(D),$ some $\delta_0=\delta_0(y_0)$ and small
enough $\varepsilon>0,$ and
$\int\limits_{0}^{\delta_0}\frac{dt}{tq_{y_0}^{\frac{1}{n-1}}(t)}=\infty$
for every $y_0\in f(D)$ and some $\delta_0=\delta_0(y_0).$

5) $p\in (n-1, n)$ and a function $Q$ satisfies $Q\in
L_{loc}^s({\Bbb R}^n)$ for some $s\geqslant\frac{n}{n-p}.$

Then $f$ is open and discrete.}
\end{corollary}

\medskip
\begin{corollary}\label{cor5}
In particular, the conclusion of Corollary \ref{cor4} holds whenever
$f\in W_{loc}^{1, p}$ instead of $f\in ACP_p.$
\end{corollary}

\medskip Generally speaking, the condition of preservation of orientation of the mapping $f$ in all statements
presented above cannot be omitted. An example of a mapping $f$ with
finite distortion of length that does not preserve orientation and
is such that $M(f(\Gamma))=M(\Gamma),$ i.e., $Q\equiv 1,$ in
inequality $Q\equiv 1,$, but is neither discrete nor open is given
in \cite[section~8.10, Ch.~8]{MRSY$_2$}.

\medskip
We also give another example. Let $x=(x_1,\ldots, x_n).$ We define
$f$ as the identical mapping in the closed domain $\{x_n\geqslant
0\}$ and set $f(x)=(x_1,\ldots,-x_n)$ for $x_n<0.$ Note that $f$ is
a mapping $f$ preserves the lengths of curves. Therefore, $f$
satisfies inequality (\ref{eq2*A}) for $Q\equiv 1.$ This mapping is
discrete but not open. For example, under the mapping $f$ the ball
${\Bbb B}^n$ is mapped into the semisphere $\{y=(y_1,\ldots, y_n)\in
{\Bbb R}^n: |y|<1, y_i\geqslant 0\},$ which is not an open set in
${\Bbb R}^n.$

\medskip
\begin{remark}\label{rem4}
Results obtained in the paper can be applied to various classes of
plane and space mappings (see \cite{MRSY$_2$} and \cite{GRSY}).

We mainly deal with a case $p\in (n-1, n].$ Unfortunately, we can
not give any conclusion about discreteness and openness of mappings
satisfying same modulus relations at arbitrary $p\geqslant 1.$
\end{remark}


{\bf \noindent Evgeny Sevost'yanov} \\
Zhitomir State University,  \\
40 Bol'shaya Berdichevskaya Str., 10 008  Zhitomir, UKRAINE \\
Phone: +38 -- (066) -- 959 50 34, \\
Email: esevostyanov2009@mail.ru
\end{document}